\documentclass[a4paper]{amsart}

\usepackage{amssymb,amsmath,amsthm}
\usepackage{mathrsfs}
\usepackage{graphicx}
\usepackage{wrapfig}
\usepackage{caption}
\usepackage{subcaption}
\usepackage{enumerate}
\usepackage[all]{xypic}
 \usepackage{url}

\newcommand{\tax}{\textup{\textit{T}\thinspace}}
\newcommand{\pn}{\textup{\textbf{pn1}\thinspace}}
\newcommand{\pnn}{\textup{\textbf{pn2}\thinspace}}
\newcommand{\pncan}{\textup{\textbf{can-pn1}\thinspace}}
\newcommand{\system}{\textup{\textbf{mIT}\thinspace}}
\newcommand{\ipc}{\textup{\textbf{IPC}\thinspace}}
\newcommand{\kax}{\textup{\textit{K}\thinspace}}
\newcommand{\dax}{\textup{\textit{D}\thinspace}}
\newcommand{\modpon}{\textup{\textbf{MP}\thinspace}}
\newcommand{\ext}{\textup{\textbf{RE}\thinspace}}
\newcommand{\mon}{\textup{\textbf{RM}\thinspace}}
\newcommand{\sfour}{\textup{\textbf{S4}\thinspace}}
\newcommand{\four}{\textup{\textit{4}\thinspace}}

\title[Simple example of weak modal logic based on intuitionistic core]{Simple example \\ of weak modal logic \\ based on intuitionistic core}
\author{Tomasz Witczak}

\address{Institute of Mathematics\\ University of Silesia\\ Bankowa~14\\ 40-007 Katowice\\ Poland}

\email{tm.witczak@gmail.com}

\date{}

\theoremstyle{Theorem}
\newtheorem{tw}{Theorem}[section]

\theoremstyle{Lemma}
\newtheorem{lem}[tw]{Lemma}

\theoremstyle{Remark}
\newtheorem{rem}[tw]{Remark}

\theoremstyle{Definition}
\newtheorem{df}[tw]{Definition}

\theoremstyle{Remark}

\usepackage{graphicx}

\begin{document}

\maketitle

\begin{abstract}
In this paper we present simple example of propositional logic which has one modal operator and is based on intuitionistic core. This system is very weak in modal sense - e.g. rules of regularity or monotonicity do not hold. It has complete semantics composed of possible worlds equipped with neighborhoods and pre-order relation. We discuss certain restrictions imposed on those structures. Also, we present characterization of axiom \four known from logic \sfour.
\end{abstract}

\section{Introduction}

Intuitionistic modal logics are often interpreted in terms of bi-relational structures. Such frames contain two relations between possible worlds: $\leq$ (pre-order, responsible for the intuitionistic aspect of logic) and $R$ (modal reachability). This approach was widely investigated in \cite{bozic}, \cite{simpson} and \cite{plotkin}. In \cite{atw} we established sound and complete semantics based on the notion of neighborhoods. In fact, we modified typical neighborhood semantics for intuitionism (presented in \cite{moniri}) by removing \textit{superset axiom}. Thus, we could speak not only about \textit{minimal} neighborhoods of worlds - but also about \textit{maximal} ones. We assumed that $w \Vdash \Box \varphi$ \textit{iff} $\varphi$ is satisfied in each world of maximal $w$-neighborhood. This definition (similar to the one used by Kojima in \cite{kojima}) was closely related with well-known relational definition of necessity (both for classical and intuitionistic systems) which states that $w \Vdash \Box \varphi$ \textit{iff} $\varphi$ is satisfied in each world visible from $w$ in modal sense.

In the present paper our approach is different - close to the typical neighborhood definition. Thus we assume (among other conditions) that formula $\Box \varphi$ is forced in $w$ \textit{iff} the whole $V(\varphi)$ belongs to the family of $w$-neighborhoods. In our earlier research we could quite easily transform neighborhood structures into bi-relational frames. Now we do not even \textit{expect} such duality - because our aim is to point out those features of neighborhoods which cannot be simulated (at least in any easy way) by relational structures.

At the same time our system is very weak in its modal aspect. It is because in classical setting neighborhoods show their usefulness especially in the universe of non-normal logics (see \cite{pacuit} for longer discussion). In some sense we \textit{isolate} axiom \tax. Quite interesting question was how to define meaning of $\Box \varphi$ and how combine $\leq$ with neighborhoods to obtain intuitionistic monotonicity of forcing. We show certain simple solution and some of its possible modifications. 

Recently, we found out that investigations in the same field are provided by Dalmonte, Grellois and Olivetti (\cite{dalmonte}). Some of their intuitions and concepts are similar to ours but there are certain differences. For example, our calculi is essentialy mono-modal (and our reflexions upon possibility operator are only secondary). Moreover, our logic always contains axiom \tax. As for the semantics, we did not use "two-dimensional" neighborhoods but only typical ones. 

\section{Alphabet and language}

Below we list basic components of our language:

\begin{enumerate}
\item $PV$ is a fixed denumerable set of propositional variables $p, q, r, s, ...$
\item Logical connectives and operators are $\land$, $\lor$, $\rightarrow$, $\bot$ and $\Box$.
\item The only derived connective is $\lnot$ (which means that $\lnot \varphi$ is a shortcut for $\varphi \rightarrow \bot$).
\end{enumerate}

Formulas are generated recursively in a standard manner: if $\varphi$, $\psi$ are \textit{wff's} then also $\varphi \lor \psi$, $\varphi \land \psi$, $\varphi \rightarrow \psi$ and $\Box \varphi$. Semantical interpretation of this language will be presented in the next section. Attention: $\Leftarrow, \Rightarrow$ and $\Leftrightarrow$ are used only on the level of meta-language (which is classical).

\section{Intuitionistic neighborhood semantics}

\subsection{The definition of frame}
Our basic structure is a pre-ordered neighborhood frame for intuitionistic modal logic (\pn-frame) defined as it follows:

\begin{df}
\label{pndef}

\pn-frame is a tripe $\langle W, \mathcal{N}, \leq \rangle$ where $\leq$ is a partial order on $W$ and $\mathcal{N}$ is a function from $W$ into $P(P(W))$ such that:

\begin{equation}
\label{cond}
w \leq v, v \in X \subseteq W, X \in \mathcal{N}_{w} \Rightarrow X \in \mathcal{N}_{v}
\end{equation}

\end{df}

\subsection{Valuation and model}
Having frame, we can establish model:

\begin{df}
\label{pnmodeldef}

\pn-model is a quadruple $\langle W, \mathcal{N}, \leq, V \rangle$ where $\langle W, \mathcal{N}, \leq \rangle$ is \pn-frame and $V$ is a function from $PV$ into $P(W)$ such that: if $w \in V(q)$ and $w \leq v$ then $v \in V(q)$.

\end{df}

\begin{df}
For every \pn-model $M = \langle W, \mathcal{N}, \leq, V \rangle$, forcing of formulas in a world $w \in W$ is defined inductively:

\begin{enumerate}

\item $w \nVdash \bot$

\item $w \Vdash q$ $\Leftrightarrow$ $w \in V(q)$ for any $q \in PV$

\item $w \Vdash \varphi \lor \psi$ $\Leftrightarrow$ $w \Vdash \varphi$ or $w \Vdash \psi$

\item $w \Vdash \varphi \land \psi$ $\Leftrightarrow$ $w \Vdash \varphi$ and $w \Vdash \psi$

\item $w \Vdash \varphi \rightarrow \psi$ $\Leftrightarrow$ $v \nVdash \varphi$ or $v \Vdash \psi$ for each $v \in W$ such that $w \leq v$

\item $w \Vdash \Box \varphi$ $\Leftrightarrow$ $w \Vdash \varphi$ and $\{z \in W; z \Vdash \varphi\} \in \mathcal{N}_{w}$

\end{enumerate}

\end{df}

As we can see, $\Box \varphi$ is forced if $V(\varphi) \in \mathcal{N}_{w}$ (just like in the classical approach) but we also require that $\varphi$ should be satisfied in an intuitionistic sense. Thus, we obtain theorem about monotonicity of forcing:

\begin{tw}
In every \pn-model $M = \langle W, \mathcal{N}, \leq, V \rangle$: if $w \Vdash \varphi$ and $w \leq v$, then $v \Vdash \varphi$.
\end{tw}

\begin{proof}
The proof goes by induction over the complexity of formulas. Almost all cases are easy (or rather just like in standard intuitionistic calculus). Thus, we shall discuss only the modal case.

Assume that $\varphi = \Box \gamma$. Suppose that $w \Vdash \varphi$. Thus $w \Vdash \gamma$ and $\{z \in W; z \Vdash \gamma\} \in \mathcal{N}_{w}$. Let us take $v \in W$ such that $w \leq v$. Of course $v \Vdash \varphi$ (by induction hypothesis). Thus, $v \in \{z \in W; z \Vdash \gamma\}$ - which means that $v \in X \in \mathcal{N}_{w}$. From (\ref{cond}) we can say that $X = V(\gamma) \in \mathcal{N}_{v}$. Hence, $v \Vdash \Box \gamma$ and finally $v \Vdash \varphi$.

\end{proof}

\subsection{About bi-relational approach}

In general, there is no universal standard of bi-relational approach for intuitionistic modal logics. There are however some clues or some popular and widely accepted notions. For example, many authors agree that forcing of $\Box \varphi$ should be defined in a following way:

$w \Vdash \Box \varphi \Leftrightarrow $ for each $v \in W$ such that $w R v$, $v \Vdash \varphi$

It means that formula $\varphi$ should be accepted in each world which is reachable from $w$ by means of relation $R$ (modal \textit{reachability} or \textit{visibility}). One could say that in our system we should use another, slightly more complicated clause (to hold connection with our basic definition):

$w \Vdash \Box \varphi \Leftrightarrow w \Vdash \varphi$ and for each $v \in W$ such that $w R v$, $v \Vdash \varphi$

Be as it may, we can show that in general it is impossible to transform an arbitrary \pn-model into the bi-relational one. For this, we use simple argument presented by Pacuit in \cite{pacuit} for classical modal setting. Suppose that we have \pn-model $M = \{W, \mathcal{N}, \leq, V\}$ such that: $W = \{w, v\}, \mathcal{N}_{w} = \{ \{w \} \}, \mathcal{N}_{v} = \{ \{v \} \}, V(\varphi) = \{w, v\}, V(\psi) = \{w\}$ and $\leq$ is empty. Now we can say that $w \Vdash \Box (\varphi \land \psi)$ - because $w \Vdash \varphi \land \psi$ and $\{z \in W; z \Vdash \varphi \land \psi\} = \{w\} \in \mathcal{N}_{w}$. At the same time, $w \nVdash \Box \varphi$ because $\{z \in W; z \Vdash \varphi\} = \{w, v\} \notin \mathcal{N}_{w}$.

Assume now that we leave our worlds and valuation without any changes but we have certain modal relation $R \subseteq W \times W$ and we use second clause to define forcing of necessity. We want to say that $w \nVdash_{R} \varphi$ (while $w \Vdash_{R} \Box (\varphi \land \psi)$). It gives us disjunction of two options. First, $w \nVdash_{R} \varphi$. Then $w \nVdash_{R} \varphi \land \psi$ and thus $w \nVdash_{R} \Box (\varphi \land \psi)$. Contradiction. Second, there is $z \in W$ such that $w R z$ and $z \nVdash_{R} \varphi$. Suppose that $z$ is $w$. Then we repeat earlier reasoning. So, check $z = v$. If $v \nVdash_{R} \varphi$ and $w R v$ then we cannot say that conjunction $\varphi \land \psi$ is accepted in each world $R$-visible from $w$. Thus, $w \nVdash_{R} \Box (\varphi \land \psi)$.


Of course this example is in fact classical - but undoubtedly it works fine with our intuitionistic definitions.

\section{Axiomatization of our system}

In this section we present sound axiomatization of our logic. Below we show its components:

\begin{df}
\label{axio}
The \system-logic is the following set of formulas and rules: $\ipc \cup \{\tax, \modpon\}$ where:

\begin{enumerate}

\item \ipc is the set of all intuitionistic axiom schemes and their modal instances

\item \tax is the axiom scheme $\Box \varphi \rightarrow \varphi$

\item \modpon is \textit{modus ponens}: $\varphi, \varphi \vdash \psi \vdash \psi$

\item \ext is \textit{rule of extensionality}: $\varphi \leftrightarrow \psi \vdash \Box \varphi \leftrightarrow \Box \psi$

\end{enumerate}
\end{df}

The following theorem holds:

\begin{tw}
\label{sound}
\system is sound with respect to the class of all \pn-models.
\end{tw}

\begin{proof}

It is easy to check that axioms and rules of \ipc are satisfied. Let us check axiom \tax. Suppose that there is \pn-model $M$ with $w \in W$ such that $w \nVdash \tax$. Thus we have $v \in W$, $w \leq v$ such that $v \Vdash \Box \varphi$ but $v \nVdash \varphi$ for certain $\varphi$. But from the definition of forcing we have immediately that $v \Vdash \varphi$. Contradiction.

\end{proof}

It can be fruitful to show explicitly that some well-known axioms and rules do not hold in our structures:

\begin{enumerate}

\item \kax: $\Box(\varphi \rightarrow \psi) \rightarrow (\Box \varphi \rightarrow \Box \psi)$

Let us consider the following model $M = \langle W, \mathcal{N}, \leq, V \rangle$: $W = \{s, c, v \}, v \leq s, \mathcal{N}_{v} = \{ \{s, c, v\} \}, \mathcal{N}_{s} = \{ \{s\}, \{s, c, v\}\}, \mathcal{N}_{c} = \{ \{c\} \}, V(\varphi) = \{s\}$ and $V(\psi) = \{s, c\}$.

Now: $v \Vdash \Box (\varphi \rightarrow \psi)$ because $v \Vdash \varphi \rightarrow \psi$ (note that $v \nVdash \varphi$ and $s \Vdash \varphi \rightarrow \psi)$ and $\{z \in W; z \Vdash \varphi \rightarrow \psi\} = \{s, c, v\} \in \mathcal{N}_{v}$. On the other side, $v \nVdash \Box \varphi \rightarrow \Box \psi$. It is because we have $s$, $v \leq s$ such that $s \Vdash \Box \varphi$ and $s \nVdash \Box \psi$. The last thing comes from the fact that $\{s, c\} \notin \mathcal{N}_{s}$.

\item \mon (\textit{rule of monotonicity}): $\varphi \rightarrow \psi \vdash \Box \varphi \rightarrow \Box \psi$

Consider the following model: $W = \{w, v\}, \mathcal{N}_{w} = \{ \{w\} \}, \mathcal{N}_{v} = \{ \{v\} \}, V(\varphi= \{w\}$ and $V(\psi) = \{w, v\}$. Of course $w \Vdash \varphi \rightarrow \psi$ and the same for $v$. Now $w \Vdash \Box \varphi$ because $w \Vdash \varphi$ and $\{z \in W; z \Vdash \varphi\} = \{v\} \in \mathcal{N}_{v}$. At the same time, $w \nVdash \Box \psi$ because $\{z \in W; z \Vdash \psi\} = \{v, z\} \notin \mathcal{N}_{v}$.
\end{enumerate}

What is interesting, is the fact that \dax axiom holds, just like in classical case. Recall that \dax is defined as $\Box \varphi \rightarrow \lnot \Box \lnot \varphi$. Suppose now that there exist \pn-model $M$ such that $w \in W$ and $w \nVdash \dax$. Thus we have $v \in W, w \leq v$ such that $v \Vdash \Box \varphi$ and $v \nVdash \lnot \Box \lnot \varphi$. Now $v \Vdash \varphi$ and $\{z \in W; z \Vdash \varphi\} \in \mathcal{N}_{v}$ but there is also $s \in W, v \leq s$ such that $s \Vdash \Box \lnot \varphi$. This means that $s \Vdash \lnot \varphi$ and $\{z \in W; z \Vdash \lnot \varphi\} \in \mathcal{N}_{s}$. But if $v \Vdash \varphi$ then of course $s \Vdash \varphi$. Contradiction.

\section{Completeness and canonical model}

In this section we prove completeness of the system \system with respect to the class of all \pn-frames. At first, we introduce certain basic definitions and lemmas.

\subsection{Useful lemmas}

\begin{df}
\system-theory is a set of well-formed formulas which contains all axioms and is closed under deduction.
\end{df}

Attention: later we shall omit symbols \system and \pn for convenience. The next lemma is quite standard and can be considered as a semantic version of deduction theorem.

\begin{lem}(see \cite{kojima}, Lemma A.1)
\label{deduct}

If $w$ is a theory then $\varphi \rightarrow \psi \in w$ $\Leftrightarrow$ $\psi \in v$ for all theories $v$ such that $w \cup \{\varphi\} \subseteq v$.

\end{lem}

\begin{proof}\textit{(sketch)}
The proof is easy. $\Rightarrow$ direction requires only \modpon rule and $\Leftarrow$ is based on the analysis of set $v = \{\psi; \varphi \rightarrow \psi \in w\}$ and axiom $\mu \rightarrow (\varphi \rightarrow \mu)$.
\end{proof}

In the next point we introduce the notion of \textit{prime} (or \textit{relatively maximal}) \textit {theory}, repeating standard definition from intuitionistic calculus.

\begin{df}
A theory $w$ is said to be \textit{prime} if it satisfies the following conditions:

\begin{enumerate}

\item $\varphi \lor \psi \in w$ $\Leftrightarrow$ $\varphi \in w$ or $\psi \in w$

\item $\bot \notin w$ (i.e. $w$ is consistent)

\end{enumerate}
\end{df}

\begin{lem}
\label{prime}
Each consistent theory $w_{\gamma}$ (which does not contain formula $\gamma$) can be extended to the prime theory $w'_{\gamma}$.
\end{lem}

\begin{proof}\textit{(sketch)}
The proof is rather standard and it does not require any specific features of our logic. The first thing is to use Lindenbaum's lemma (or well-known methods for countable languages) which allows us to extend $w_{\gamma}$ to the relatively maximal $w'_{\gamma}$. The second thing is to prove that $w'_{\gamma}$ is actually prime. It is enough to prove that $\gamma \in \overline{w'_{\gamma} \cup \{\varphi\}}$ (resp. $\overline{w'_{\gamma} \cup \{\psi\}}$) where by $\overline{X}$ we mean \textit{deductive closure} of the set of formulas $X$. It is important that we use semantic deduction theorem in this proof.
\end{proof}

\subsection{Canonical model}

\begin{df}
\pncan (\textit{canonical neighborhood model}) is a triple $\langle W, \mathcal{N}, \leq, V \rangle$ where:

\begin{enumerate}

\item $W$ is the set of all prime theories

\item for every $w, v \in W$ we say that $w \leq v$ \textit{iff} $w \subseteq v$

\item $\mathcal{N}$ is a function from $W$ into $P(P(W))$ such that for every $w \in W$ and for each formula $\varphi$: $\mathcal{N}_{w} = \{ \{z \in W; \varphi \in z\}; \Box \varphi \in w \}$.

\item $V: PV \rightarrow P(W)$ is a function defined as it follows: $w \in V(q) \Leftrightarrow q \in w$

\end{enumerate}

\end{df}

Note that we can say: 

\begin{rem}
\label{uwaga}
In each \pncan-model: if $X \subseteq W$, then $X \in \mathcal{N}_{w} \Leftrightarrow$ there is a formula $\varphi$ such that $X = \{z \in W; \varphi \in z\}$ and $\Box \varphi \in w$. 
\end{rem}

\begin{rem}
\label{uwaga2}
In each \pncan-model: $\{z \in W; \varphi \in z\} \in \mathcal{N}_{w} \Leftrightarrow \Box \varphi \in w$.
\end{rem}

One could ask (see \cite{pacuit}) if it possible that $\{z \in W; \varphi \in z\} = \{z \in W; \psi \in z\}$ and $\Box \varphi \in w$ but $\Box \psi \notin w$. Surely, it would spoil our definition. Thus we prove the following two lemmas:

\begin{lem}
\label{equ}
Assume that $W$ is a collection of all prime theories of \system. Suppose now that $\{z \in W; \varphi \in z\} = \{z \in W; \psi \in z\}$. Then $\vdash \varphi \leftrightarrow \psi$, i.e. $\varphi \leftrightarrow \psi \in \system$. 
\end{lem}

\begin{proof}
Suppose that $\nvdash \varphi \leftrightarrow \psi$. We can assume (without loss of generality) that $\nvdash \varphi \rightarrow \psi$, so $\varphi \rightarrow \psi \notin \system$. From lemma \ref{deduct} there exists theory $v$ such that $\varphi \in v$ but $\psi \notin v$. By means of lemma \ref{prime} we can expand $v$ to the prime theory $t$ such that $\psi \notin v$ (of course, $\varphi \in t$). But now $t \in \{z \in W; \varphi \in z\}$ - so, as we assumed, $t \in \{z \in W; \psi \in z\}$. Contradiction.
\end{proof}

\begin{lem}
In \pncan-model $\langle W, \mathcal{N}, \leq, V \rangle$ we have the following property: for each prime theory $w$, if $\{z \in W; \varphi \in z\} \in \mathcal{N}_{w}$ and $\{z \in W; \varphi \in z\} = \{z \in W; \psi \in z\}$, then $\Box \psi \in w$.
\end{lem}

\begin{proof} (see \cite{pacuit})

If $\{z \in W; \varphi \in z\} \in \mathcal{N}_{w}$ then (from remark \ref{uwaga2})  $\Box \varphi \in w$. From lemma \ref{equ} we have that $\vdash \varphi \leftrightarrow \psi$. Thus, $\varphi \leftrightarrow \psi \in \system$. By \ext, also $\Box \varphi \leftrightarrow \Box \psi \in \system$. In particular, it means that $\Box \varphi \leftrightarrow \Box \psi \in w$. Hence, $\Box \psi \in w$. 

\end{proof}

Now we have the following lemma:

\begin{lem}
\pncan is a well-defined \pn-model.
\end{lem}

\begin{proof}
What is really important to check, is the relation between $\leq$ and neighborhoods. Suppose that $w \subseteq v$ and $v \in X \in \mathcal{N}_{w}$. By definition, $X = \{z \in W; \varphi \in z\}$ for certain $\varphi$ such that $\Box \varphi \in w$. But if $w \subseteq v$, then $\Box \varphi \in v$. Now we recall remark \ref{uwaga} to obtain the final result, that is: $\{z \in W; \varphi \in z\} \in \mathcal{N}_{v}$.
\end{proof}

Below is the crucial lemma:

\begin{lem}(\textit{truth lemma})
\label{truth}
In \pncan-model we have for each $\psi$ and for each $w \in W$: $w \Vdash \varphi \Leftrightarrow \varphi \in w$.
\end{lem}

\begin{proof}
The proof goes by induction on the complexity of formula. There is one non-trivial case where neighborhoods are involved: that of $\varphi := \Box \psi$.
($\Rightarrow)$. Suppose that $w \Vdash \Box \psi$. Thus $w \Vdash \psi$ and $\{z \in W; z \Vdash \psi\} \in \mathcal{N}_{w}$. By induction, $\psi \in w$ and $\{z \in W; \psi \in z\} \in \mathcal{N}_{w}$. The last statement means (by remark \ref{uwaga2}) that $\Box \psi \in w$.

($\Leftarrow$). Assume that $\Box \varphi \in w$. From \tax we have that $\varphi \in w$. Now by the definition of $\mathcal{N}$ we can say that $\{z \in W; \varphi \in z\} \in w$. But then (by induction hypothesis) $w \Vdash \varphi$ and $\{z \in W; z \Vdash \varphi\} \in \mathcal{N}_{w}$. Thus $w \Vdash \Box \varphi$.
\end{proof}

\begin{tw}\system is complete with respect to the class of all \pn-frames.
\end{tw}

\begin{proof}
Suppose that $w$ is a theory and $w \nvdash \varphi$. In particular this means that $\varphi \notin w$. Then we can extend $w$ to the prime theory $v$ such that $w \subseteq v$ and $\varphi \notin v$. Of course for each $\psi \in w$, we have $\psi \in v$. Now we use lemma \ref{truth} to say that $v \Vdash \psi$ and $v \nVdash \varphi$. The last statement means in particular that $\varphi$ is not a semantical consequence of $w$.
\end{proof}

\subsection{Additional restriction}

Let us consider the following condition imposed on our models:

\begin{equation}
\label{cond2}
w \leq v \Rightarrow \mathcal{N}_{w} \subseteq \mathcal{N}_{v}
\end{equation}

This restriction is stronger than (\ref{cond}). We can prove the following lemma:

\begin{lem}
In \pncan-model condition (\ref{cond2}) is valid.
\end{lem}

\begin{proof}
Suppose that we have prime theories $w, v$ such that $w \subseteq v$. Consider an arbitrary $X \in \mathcal{N}_{w}$. Of course, by the definition of neighborhood in canonical model, $X = \{z \in W; \varphi \in z\}$ for certain $\varphi$ such that $\Box \varphi \in w$. If $w \subseteq v$, then $\Box \varphi \in v$. Thus $X \in \mathcal{N}_{v}$.
\end{proof}

The last result means that we can limit completeness of \system to the class of all \pn-frames satisfying (\ref{cond2}). We shall call them \pnn-frames (models). 

One could say that in the presence of (\ref{cond2}) we may simplify our definition of forcing - without violating monotonicity. In fact, we can introduce the following definition: 

$w \Vdash \Box \varphi \Leftrightarrow \{z \in W; z \Vdash \varphi\} \in \mathcal{N}_{w}$

This approach is identical with the classical one. Note, however, that now we have different logic which is not equivalent with \system. For example, axiom $T$ does not hold. This system is just intuitionism with modal rule of extensionality.

\section{Possibility operator}

In this section we work with \pnn-frames. Our goal is to establish sensible notion of possibility operator $\nabla$. We propose the following definition:

$w \Vdash \nabla \varphi \Leftrightarrow$ there are $X \in \mathcal{N}_{w}$ and $z \in X$ such that $z \Vdash \varphi$

\begin{tw}
In every \pnn-model $M = \langle W, \mathcal{N}, \leq, V \rangle$: if $w \Vdash \varphi$ and $w \leq v$, then $v \Vdash \varphi$.
\end{tw}

\begin{proof}
Of course \pnn-models are subclass of \pn-models. Thus we do not check monotonicity for $\land, \lor, \rightarrow$ and $\Box$. But let us assume that $w \Vdash \nabla \gamma$ and $w \leq v$. Then there is $X \in \mathcal{N}_{w}$ such that for certain $z \in X$ we have $z \Vdash \gamma$. Now $X \in \mathcal{N}_{v}$ so we can say that $v \Vdash \nabla \varphi$.
\end{proof}

Note that we have explicitly used the fact that $\mathcal{N}_{w} \subseteq \mathcal{N}_{v}$. One can check that monotonicity of forcing holds (in \pnn-models) also with the following interpretation of possibility:

$w \Vdash \Diamond \varphi \Leftrightarrow $ for each $X \in \mathcal{N}_{w}$ there is $z \in X$ such that $z \Vdash \varphi$

However, the second approach is quite problematic. While $\Box \varphi \rightarrow \nabla \varphi$ is true, then we cannot say the same about $\Box \varphi \rightarrow \Diamond \varphi$. Roughly speaking, $\Box$ guarantees us that the set of all worlds satisfying $\varphi$ is one of the $w$-neighborhoods. But it does not guarantee that in \textit{each} $w$-neighborhood we shall find world satisfying $\varphi$.

\section{Question of axiom 4}

Axiom \four (i. e. $\Box \varphi \rightarrow \Box \Box \varphi$) is typical for propositional system \sfour, introduced by Lewis. In standard neighborhood setting for classical modal logics this formula corresponds to the following condition (see \cite{indrze}):

\begin{equation} \label{cztery}
X \in \mathcal{N}_{w} \Rightarrow \{v \in W; X \in \mathcal{N}_{v}\} \in \mathcal{N}_{w} \tag{$\star$}
\end{equation}

We shall show that this restriction is too weak for characterization of \four in our environment. Let us consider the following \pn-model $M = \{W, \leq, \mathcal{N}, V\}$:

$W = \{v, z, u\}, \mathcal{N}_{v} = \{ \{u, v\}, \{v, z\} \}, \mathcal{N}_{z} = \{ \{u, v\}, \{v, z\} \}, \mathcal{N}_{u} = \{ \{u\} \}, V(\varphi) = \{v, u\}$

One can easily check that $M$ is a proper \pn-model which satisfies \eqref{cztery}. Now we can say that $v \Vdash \Box \varphi$ because $v \Vdash \varphi$ and $\{x \in W; x \Vdash \varphi\} = \{v, u \} \in \mathcal{N}_{v}$. On the other hand, $v \nVdash \Box \Box \varphi$ because $\{x \in W; x \Vdash \Box \varphi\} = \{v \} \notin \mathcal{N}_{v}$.

For this reason we have found another characterization:

\begin{lem}
Axiom \four holds in \pn-model $M$ \textit{iff} $M$ satisfies the following condition:

\begin{equation}\label{czterydwa}
X \in \mathcal{N}_{w} \Rightarrow \left(Y \subseteq X \Rightarrow \left(Y \in \mathcal{N}_{w} \right) \right) \tag{$\star \star$}
\end{equation}

\end{lem}

\begin{proof}

Assume that \pn-model $M = \{W, \leq, \mathcal{N}, V \}$ satisfies \eqref{czterydwa} and there is $w \in W$ such that $w \Vdash \Box \varphi$. Hence, $w \Vdash \varphi$ and $X = \{x \in W; x \Vdash \varphi\} \in \mathcal{N}_{w}$. But $Y = \{x \in W; x \Vdash \Box \varphi\} \subseteq X$ - and thus $Y \in \mathcal{N}_{w}$. So by the definition of forcing $w \Vdash \Box \Box \varphi$.

As for the other direction, we can use earlier counter-example. Clearly, it does not satisfy \eqref{czterydwa}.

\end{proof}

\section{Further investigations}

This paper should be considered only as a short introduction into research of weak modal logics based on intuitionistic core. There are still many opened questions. For example, it would be interesting to obtain completeness results for weak bi-modal intuitionistic logics, i.e. with possibility operator (defined as here or in a different way). Also, it would be fruitful to characterize various frame conditions by means of formulas (still with completeness). What is important from our point of view, is to use additional tools (axioms, restrictions on frames etc.) without going "too far". In other words, we do not want "too strong" logics (even if notions of modal "weakness" and "strongness" are somewhat unclear or arbitrary). Finally, there is also another interesting task: to combine modalities (and neighborhoods) with subintuitionistic systems (in non-trivial way). As far as we know, the area of \textit{subintuitionistic modal logics} is almost \textit{terra incognita}.

\end{document}